\documentclass[11pt]{amsart}
\usepackage{amsmath,amsfonts,amssymb,amsthm}
\usepackage[alphabetic]{amsrefs}
\usepackage[foot]{amsaddr}
\usepackage{hyperref}
\usepackage{graphicx,color}
\usepackage{pictexwd,dcpic,epsf}
\usepackage{enumerate}

\newtheorem{theorem}{Theorem}
\newtheorem{lemma}[theorem]{Lemma}

\newtheorem{corollary}{Corollary}

\theoremstyle{definition} 
\newtheorem*{definition}{Definition}

\theoremstyle{remark}

\newcommand {\mr}{\mathrm}

\newcommand{\ts}{\widetilde s}
\newcommand{\tg}{\widetilde g}
\newcommand{\tS}{\widetilde S}
\newcommand{\tGa}{\widetilde{\Gamma}}
\newcommand{\tG}{\widetilde{G}}
\newcommand{\tv}{\widetilde v}
\newcommand{\tone}{\widetilde{1}}

\begin{document}

\title[
Group cubization
]{
Group cubization
}

\author{Damian Osajda\\
  \\	{\tiny{\it{with an appendix by Mika\"el Pichot}}}}
\address{Instytut Matematyczny,
Uniwersytet Wroc\l awski\\
pl.\ Grun\-wal\-dzki 2/4,
50--384 Wroc{\l}aw, Poland}
\address{Institute of Mathematics, Polish Academy of Sciences\\
\'Sniadeckich 8, 00-656 War\-sza\-wa, Poland}
\email{dosaj@math.uni.wroc.pl}

\begin{abstract}
We present a procedure of group cubization: it results in a group 
whose some features resemble the ones of a given group,
and which acts without fixed points on a CAT(0) cubical complex. As a main application we
establish lack of Kazhdan's property (T) for Burnside groups. 
\end{abstract}

\maketitle

\section{Introduction}
\label{s:intro}

The initial motivation for the current article was the following well-known question:

{\begin{center}\emph{Do all Burnside groups have Kazhdan's property (T)?}
\end{center}}
\smallskip

\noindent
(See e.g.\ \cite[Open Problem 17, page 9]{Ober}, the book \cite[Open Example 7.3, page 282]{BdlHV}, or 
\cite[page 1304]{Sha}, and \cite[Conjecture]{Sha2}. 
In the latter two, Y.\ Shalom presents it as a conjecture and provides a motivation.)

Recall that infinite \emph{Burnside groups}, that is, finitely generated groups of bounded torsion were first
constructed by Novikov-Adyan \cite{AdNov}. The \emph{free Burnside group} $B(m,n)$ is defined by the 
presentation \[ \langle s_1,s_2,\ldots,s_m \; | \; w(s_1,s_2,\ldots,s_m)^n \rangle,\] where $w$ runs over
all words in $s_1,s_2,\ldots, s_m$. We answer in the negative the above question by proving the following.

\begin{theorem}
\label{t:1}
	If the free Burnside group $B(m,n)$ is infinite then, for every integer $k>1$, the free Burnside group $B(m,kn)$ 
	acts without fixed points on a CAT(0) cubical complex and hence does not have Kazhdan's property (T).
\end{theorem}

As the main tool we introduce a general procedure called \emph{group cubization}.\footnote{We chose the term
 ``cubization" as an analogue of ``hyperbolization" \cite{Gro}. Both procedures modify the object while preserving
 some of its features. In contrast, ``cubulation" (see e.g.\ \cites{W-qch,W-book}) equips a given object with an additional structure.} It is a simple trick interesting on its own,
that we believe may be of broad use. It works as follows.

Let $G$ be a finitely generated group. Let $\tGa$ be the $\mathbb Z_k$--homology cover of 
a Cayley graph $\Gamma$ of $G$, where $\mathbb{Z}_k$ is the group $\mathbb Z/k\mathbb Z$ of integers modulo $k$.
The covering graph $\tGa$ is equipped with a structure of a space with walls, defined by preimages
of edges in $\Gamma$. We prove the following.

\begin{theorem}
\label{t:2}
	For a finitely generated group $G$ and its Cayley graph $\Gamma$, the $\mathbb Z_k$--homology cover 
	$\tGa$ of $\Gamma$ is a Cayley graph of a finitely generated group $\tG$. If $G$ is infinite then $\tG$ acts with unbounded orbits on a CAT(0) cubical complex.
\end{theorem}

The group $\tG$ is called a \emph{cubization} of $G$.
Theorem~\ref{t:1} follows easily from Theorem~\ref{t:2} since a cubization $\tG$ of a Burnside group
$G$ is a Burnside group (of exponent multiplied by $k$).
\medskip

In the rest of the paper, after some preliminaries (Section~\ref{s:Cay}) we prove (in Section~\ref{s:12}) Theorem~\ref{t:2}, then Theorem~\ref{t:1}, and finally we make some remarks on further applications of group cubizations (Section~\ref{s:app}).
\medskip

In an appendix by Mika\"el Pichot an alternative proof of a result similar to Theorem~\ref{t:1} is presented.
\medskip

\noindent
{\bf Acknowledgments.} I thank \'Swiatos\l aw R.\ Gal, Tadeusz Januszkiewicz, Damian Sawicki, Daniel T. Wise,
and the anonymous referees for remarks leading to improvements in the article.
Parts of the paper were written while visiting McGill University.
I would like to thank the Department of Mathematics and Statistics of McGill University
for its hospitality during that stay.
The research was partially supported by (Polish) Narodowe Centrum Nauki, grant no.\ UMO-2015/\-18/\-M/\-ST1/\-00050.

\section{Group cubization}
\label{s:Cay}
The results in this section use only few basic facts about covering spaces. A standard reference for 
those is e.g.\ \cite{Hatcher}. 

Let $\Gamma=\mr{Cay(G,S)}$ be a Cayley graph of a group $G$ generated by a finite symmetric set $S$ (that is, $S=S^{-1}$).
We use the convention that every vertex belongs to two edges corresponding to a pair $\{s,s^{-1}\}$, for each $s\in S$. In particular,
generators being involutions give rise to double edges, and the degree of vertices is $|S|$. 
Fix an integer $k>1$. Let $p\colon \tGa \to \Gamma$ be the \emph{$\mathbb Z_k$--homology covering} of $\Gamma$, that is,
the covering corresponding to the kernel $K$ of the natural map $\pi_1(\Gamma,1)\to H_1(\Gamma;\mathbb Z_k)=
\bigoplus_I \mathbb Z_k$, where $I$ is the set indexing generators of $\pi_1(\Gamma,1)$, and $1$ is the vertex 
of $\Gamma$ being the identity of $G$.
Observe that
it is a characteristic covering, that is, $K$ is a characteristic subgroup of $\pi_1(\Gamma,1)$.
This subgroup can be identified with $\pi_1(\tGa,\tone)$, where  $\tone$ is a vertex with $p(\tone)=1$. 
Therefore, for any automorphism $g$ of $\Gamma$, we have 
$g_{\ast}\circ p_{\ast} (\pi_1(\tGa,\tone))=p_{\ast} (\pi_1(\tGa,\tone))$, and hence
$g\circ p\colon \tGa \to \Gamma$ can be lifted to a map $\tg \colon \tGa \to \tGa$ satisfying
$p\circ \tg= g \circ p$.

Every element $g\in G$ defines an automorphism of $\Gamma$ given by left multiplication by $g$ (and
also denoted by $g$).
Let $\tG$ denote the set of all lifts of all such automorphisms.

\begin{theorem}
	\label{t:cay}
	The set $\tG$ forms a group generated by a set $\tS$ with $|\tS|=|S|$, and with $\tGa$ being its Cayley
	graph $\mr{Cay}(\tG,\tS)$.
\end{theorem}
\begin{proof}
	Observe that for any automorphism $g$ of $\Gamma$ its lift is entirely determined by 
	the value on any vertex of $\tGa$. Therefore, $\tG$ is a group and acts transitively
	on $\tGa$.
	
	For $g\in G$, if its lift $\tg$ fixes a vertex $\tv \in \tGa$ then $g(p(\tv))=p(\tv)$.
	Since $G$ acts freely on $\Gamma$, it follows that $g=1$, and hence $\tg$ is a deck transformation.
	As $\tg$ fixes a vertex, it is the identity on $\tGa$. Therefore, $\tG$ acts freely on $\tGa$. 
	
	From the Sabidussi theorem \cite{Sab} it follows now that $\tGa$ is the Cayley graph $\mr{Cay}(\tG,\tS)$,
	for some $\tS$. Since the degree of vertices in $\Gamma$ is equal to the degree of vertices in $\tGa$, and they are both equal
	to cardinalities of the corresponding generating sets, we have $|S|=|\tS|$.
\end{proof}

\begin{definition}
	\label{d:cub}	
	The group $\tG$ is called the \emph{cubization} of a group $G$ with respect to its Cayley graph $\Gamma=\mr{Cay(G,S)}$.
\end{definition}

Choose a vertex $\tone \in \tGa$ with $p(\tone)=1\in \Gamma$. For every $s\in S$ we may choose its lift $\ts$  such that $\ts(\tone)$ is a vertex adjacent to $\tone$. 
The set $\tS=\{\ts \;| \;s\in S\}$ of such lifts is a generating set of $\tG$. We use it below.

\begin{lemma}
	\label{l:id}
	Let $g_1,\ldots, g_n\in G$ be such that $g_1g_2\cdots g_n=_G1$.
	For $i=1,\ldots,n$, let $\tg_i$ be a lift of $g_i$.
	Then $(\tg_1\cdots \tg_n)^k= \mr{id}_{\tGa}$.
\end{lemma}
\begin{proof}
	The set of deck transformations of $p\colon \tGa \to \Gamma$ is a group
	isomorphic to $\pi_1(\Gamma,1)/K$, i.e., to $H_1(\Gamma,\mathbb Z_k)=\bigoplus_I \mathbb Z_k$.
	Since the latter has exponent $k$, for every deck transformation $\tg$ (that 
	is when $p\circ \tg = p$)
	we have $\tg^k=\mr{id}_{\tGa}$.
	In particular, if $g_1\cdots g_n=_G1$ then $\tg_1\cdots \tg_n$ is a deck transformation, and hence $(\tg_1\cdots \tg_n)^k= \mr{id}_{\tGa}$.
\end{proof}

\section{Proofs of Theorem~\ref{t:1} and Theorem~\ref{t:2}}
\label{s:12}
\subsection{Proof of Theorem~\ref{t:2}}
By Theorem~\ref{t:cay} we have that $\tGa$ is the Cayley graph $\mr{Cay}(\tG,\tS)$.
First, we consider the case when no edge of $\Gamma$ separates $\Gamma$.
D.\ Wise \cite[Section 9]{W-qch} and \cite[Section 10.3]{W-book} 
observed that the 
vertex set of the 
$\mathbb Z_2$--homology cover of a graph has a natural structure of a space with walls.\footnote{In fact, in \cite{W-qch} the result is set in
	 the more general setting of CAT(0) cube complexes.
	The preprint was circulating since 2009, and the results were presented e.g.\ during a conference 
	at UQAM in April 2010. Independently, the same result has been shown for graphs in \cite{AGS}.
	}
Similar structure exists for the 
$\mathbb Z_k$--homology cover (see e.g.\ \cite[discussion on page 57]{Khu}):
the preimage of every open edge of $\Gamma$ disconnects the cover $\tGa$ into $k$ connected components.
Such a partition of vertices of $\tGa$ gives rise to $2^{k-1}-1$ partitions
into two nonempty sets -- these are walls in the space with walls  $(\tGa^{(0)}, \mathcal W)$.
Obviously, the group $\tG$ acts on $(\tGa^{(0)}, \mathcal W)$: for every generator $\ts \in \tS$ a wall corresponding to $e\in E$ is mapped by
$\ts$ to a wall corresponding to $s(e)$. Clearly, the action has orbits that are unbounded
with respect to the wall pseudo-metric. 

In the case when there exists an edge separating $\Gamma$, the translates (by $G$) of this edge
define the space with walls $(\Gamma^{(0)},\mathcal W')$. The $G$--action on $(\Gamma^{(0)},\mathcal W')$ has unbounded orbits, and induces a $\tG$--action on $(\Gamma^{(0)},\mathcal W')$ with the same property.

Finally, by \cites{Nica,ChaNi} it follows that
$\tG$ acts with unbounded orbits on a corresponding CAT(0) cubical complex.

\subsection{Proof of Theorem~\ref{t:1}}
Let $G=B(m,n)$ be an infinite free Burnside group.\footnote{See e.g.\ \cite{Coulon} for a description of the state of the art of the theory of infinite 
	Burnside groups.}
Let $\Gamma$ be its Cayley graph with respect to the (symmetric)
generating set $S$, with $|S|=2m$. 
By Theorem~\ref{t:cay}, the cubization $\tG$ of $G$ with respect to $\Gamma$ is a group with the Cayley graph $\tGa=\mr{Cay}(\tG,\tS)$ being the $\mathbb Z_k$--homology cover of $\Gamma$, for $\tS=\{\ts \; | \; s\in S\}$. 
For every word $w=\ts_{1}\cdots \ts_{l}$ we have that $(s_{1}\cdots s_{l})^n=_G1$ and hence, by Lemma~\ref{l:id}, $w^{kn}=_{\tG}1$. It follows that the cubization $\tG$ is a Burnside group
of exponent $kn$. 
By Theorem~\ref{t:2} the group $\tG$ acts with unbounded orbits on a CAT(0) cubical complex.
As $\tG$ is a quotient of the free Burnside group $B(m,kn)$, the latter admits a similar action.

\section{Further applications}
\label{s:app}

Theorem~\ref{t:2} may be used to establish the existence of unbounded actions on
CAT(0) cubical complexes for other classes of groups. 
In particular we have the following.

\begin{corollary}
	\label{c:app}
	If a group defined by a presentation $\langle S \; | \; r_1,r_2,\ldots \rangle$ is infinite
	then the group with the presentation $\langle S \; | \; r_1^k,r_2^k,\ldots \rangle$ acts with unbounded
	orbits on a CAT(0) cubical complex.
\end{corollary}

This applies to the class of groups defined by presentations in which relators are proper powers.
Many important classical examples of groups belong here: generalized triangle groups, generalized von Dyck groups, groups given by Coxeter's presentations of types $(l,m\, | \, n,k)$, $(l,m,n;q)$, and $G^{m,n,p}$ (see e.g.\ \cite{Tho} and references therein), or various small cancellation groups (see e.g.\ \cites{W-qch,W-book} and references
therein).


\newpage

\thispagestyle{empty}

\begin{center}
	\textbf{Appendix:}
	
	\textbf{A comment on Osajda's ``Group cubization'' paper}
	
	\bigskip
	{Mika\"el Pichot*}
	
	\bigskip
	\bigskip
\end{center}

Answering a question of Shalom, Damian Osajda proved in the main body of this paper that the free Burnside group $B(m,kn)$, for $k\geq 2$, does not have the property (T) of Kazhdan if $B(m,n)$ is infinite. 

Osajda's proof is a nice geometric argument using mod $k$ homology covers of Cayley graphs. 

There is also a more algebraic argument for a similar result, which can  implicitly be found in the literature. The goal of this appendix is to explain this alternative argument. 

Consider the  wreath product $W:=\mathbb{Z}/k\mathbb{Z} \wr B(m,n)$ of $B(m,n)$ with a finite cyclic group. Notice that the order of every element in $W$ divides $kn$, and that $W$ is generated by $m+1$ elements. Therefore, by the universal property, $W$ is a quotient of the free Burnside group $B(m+1,kn)$ of exponent $kn$. It follows from [1, Theorem 3, p. 1897] or [3, Theorem 1.2, p. 168] that the wreath product $H\wr G$ does not have Kazhdan's property (T) if $G$ is infinite and $H$ is not trivial.  
In particular, the group $B(m+1,kn)$ does not have Kazhdan's property (T) if $B(m,n)$ is infinite. (Both the wreath product approach and Osajda's mod $k$ homology cover approach provide wall spaces in the over group, and therefore prove more than just the lack of property (T).)

The fact that nontrivial wreath product decompositions appear as quotients of large free Burnside groups was used in [2, Proof of Theorem 2] in relation with Dixmier's unitarizability problem. 

The author is grateful to Nicolas Monod, Narutaka Ozawa, and Alain Valette for their comments.

\bigskip
\medskip

\begin{center}
	\textsc{References}
\end{center}

\medskip

\begin{itemize}
	\item[{[1]}] P.-A. Cherix, F. Martin, and A. Valette. ``Spaces with measured walls, the Haagerup property and property (T)''. Ergodic theory and dynamical systems 24.06 (2004), pp. 1895--1908.
	\item[{[2]}] N. Monod and N. Ozawa. ``The Dixmier problem, lamplighters and Burnside group''. Journal of Functional Analysis 258.1 (2010), pp. 255--259.
	\item[{[3]}] M. Neuhauser. ``Relative property (T) and related properties of wreath products''. Mathematische Zeitschrift 251.1 (2005), pp. 167--177.
\end{itemize}
\bigskip

(*) Dept. of Mathematics \& Statistics, McGill University, Montr\'eal, Qu\'ebec, Canada H3A 2K6. \email{pichot@math.mcgill.ca}

\end{document}